\documentclass[letterpaper,11pt]{article}

\usepackage[english]{babel}
\usepackage{amsmath}
\usepackage{enumerate}
\usepackage{multicol}
\usepackage{commath}
\usepackage{graphicx}
\usepackage[letterpaper, margin=1in]{geometry}
\usepackage{amsfonts}
\usepackage{amsthm}
\usepackage{caption}
\usepackage{subcaption}
\usepackage{mdframed}
\usepackage{amssymb}
\usepackage{xfrac}
\usepackage{mathtools}
\usepackage{bm}

\usepackage{csquotes}
\usepackage{stmaryrd}
\usepackage{tikz}
\usetikzlibrary{arrows}
\usetikzlibrary{positioning}
\usepackage{color}
\usepackage{hyperref}

\newtheorem{theorem}{Theorem}[section]
\newtheorem*{theorem*}{Theorem}
\newtheorem{proposition}[theorem]{Proposition}
\newtheorem{lemma}[theorem]{Lemma}
\theoremstyle{definition}
\newtheorem{definition}[theorem]{Definition}
\newtheorem{example}[theorem]{Example}

\usepackage{cite}

\hypersetup{
    colorlinks=true,
    linktoc=all,
    linkcolor=black,
}

\usepackage{tocloft}

\title{S-limited shifts}

\author{Benjamin Matson\\
Carleton College\\
\texttt{ben.matson@gmail.com}\\
\and
Elizabeth Sattler\\
Carleton College\\
\texttt{lsattler@carleton.edu}
}

\date{August 2017}

\begin{document}
\theoremstyle{definition}

\newenvironment{fexample}
  {\begin{mdframed}\begin{example}}
  {\end{example}\end{mdframed}}

\newcommand{\bigslant}[2]{{\raisebox{.2em}{$#1$}\left/\raisebox{-.2em}{$#2$}\right.}}

\maketitle

%\tableofcontents
%\thispagestyle{empty}
%\clearpage
\pagenumbering{arabic}

\abstract{In this paper, we explore the construction and dynamical properties of $\mathcal{S}$-limited shifts.  An $S$-limited shift is a subshift defined on a finite alphabet $\mathcal{A} = \{1, \ldots,p\}$ by a set $\mathcal{S} = \{S_1, \ldots, S_p\}$, where $S_i \subseteq \mathbb{N}$ describes the allowable lengths of blocks in which the corresponding letter may appear.  We give conditions for which an $\mathcal{S}$-limited shift is a subshift of finite type or sofic.  We give an exact formula for finding the entropy of such a shift and show that an $\mathcal{S}$-limited shift and its factors must be intrinsically ergodic.  Finally, we give some conditions for which two such shifts can be conjugate, and additional information about conjugate $\mathcal{S}$-limited shifts.}

\section{Introduction}

$S$-gap shifts are a class of shift spaces defined on the alphabet $\mathcal{A} = \{0,1\}$ by a set $S \subseteq \mathbb{N}_0$ (where $\mathbb{N}_0 = \mathbb{N} \cup \{0\})$, which describes the allowable number of 0s that can separate two 1s in a string in the space.  $S$-gap shifts and their dynamical properties have been and continue to be studied thoroughly and used in applications (\cite{Baker}, \cite{Jung}, \cite{Sgap}).  In particular, in \cite{Jung}, Jung proved that an $S$-gap shift $X(S)$ is mixing if and only if $gcd\{n+1: n \in S\}=1$, along with other necessary and sufficient conditions for which $X(S)$ satisfies various specification properties.  In \cite{Sgap}, Dastjerdi and Jangjoo established a collection of dynamical and topological properties of $S$-gap shifts, including a characterization of the $S$-gap shifts that are subshifts of finite type or sofic.  
In \cite{SPrimegap}, Dastjerdi and Jangjooye introduced a broader class of shift spaces, called $(S, S')$-gap shifts, which again are defined on the alphabet $\{0,1\}$ by two sets $S, S' \subseteq \mathbb{N}_0$, which define the allowable lengths of the blocks of 0s and 1s, respectively.  They established many properties of this class of shift spaces, including results about the entropy of an $(S,S')$-gap shift, and some specific conditions for conjugacy.
\medskip

In this paper, we investigate a broader class of shift spaces, called $\mathcal{S}$-limited shifts. These subshifts are defined on an alphabet $\{1, \ldots, p \}$  by a finite set $\mathcal{S} = \{S_1, \ldots, S_p\}$ with $S_i \subseteq \mathbb{N}$ for $1 \leq i \leq p$, and each $S_i$ describes the allowable lengths of blocks of the corresponding letter in a string in the shift space.  For the majority of the paper, we restrict the order in which the blocks may appear.  In the more general setting in which the order is unrestricted, we refer to the shift as a generalized $\mathcal{S}$-limited shift.  In both cases, we study important dynamical properties of these shift spaces.  In particular, we prove the following result about the entropy of an $\mathcal{S}$-limited shift.  

\begin{theorem*} 
Let $\mathcal{S}=\{S_1,S_2,...,S_p\}$ such that $S_i\subseteq\mathbb{N}$ for $1 \leq i \leq p$. Then, the entropy of the $\mathcal{S}$-limited shift $X(\mathcal{S})$ is $\log \lambda$, where $\lambda$ is the unique positive solution to 
\[\displaystyle\sum_{\omega \in G_\mathcal{S}}x^{-|\omega|}=1,\]
where $G_\mathcal{S}=\{1^{m_1}2^{m_2}...p^{m_p} : m_i\in S_i\text{ for }1\leq i\leq p\}.$
\end{theorem*}

We note that this theorem is an extension of a result in \cite{SPrimegap}, which states that the entropy of an $(S, S')$-gap shift is given by $\log(\lambda)$, where $\lambda$ is the unique non-negative solution to 
\[ \sum_{s+s' \in \{\!\!\{S+S'\}\!\!\} } x^{-(s+s'+2)}=1,\]
where $\{\!\!\{S+S'\}\!\!\} = \{ s+s' : s\in S, s' \in S'\}$ and values of multiplicities are included (that is, if $s_1 + s_1' = s_2 + s_2'$ but $s_1 \neq s_2$ then both $s_1 +s_1'$ and $s_2 + s_2'$ are included in $\{\!\!\{S + S'\}\!\!\}$).  
\medskip

Next, we obtain the following results regarding conjugacy between two $\mathcal{S}$-limited shifts.

\begin{theorem*}
Let $\mathcal{S}=\set{S_1,...,S_p}$ and $\mathcal{T}=\set{T_1,...,T_q}$ and suppose $X(\mathcal{S})$ is conjugate to $X(\mathcal{T})$. Then, where $G_\mathcal{S}$ and $G_\mathcal{T}$ denote the collections $\{1^{s_1}\ldots p^{s_p}: s_i \in S_i\}$ and $\{1^{t_1} \ldots q^{t_q}; t_i \in T_i\}$ respectively, for all $l\in\mathbb{N}$, $\abs{\set{x\in G_\mathcal{S} : |x|=l}}=\abs{\set{y\in G_\mathcal{T} : |y|=l}}$.
\end{theorem*}

\begin{theorem*}
	Let $\mathcal{S}=\set{S_1,...,S_p}$ and $\mathcal{T}=\set{T_1,...,T_p}$ where for all $i$, $S_i,T_i\subseteq\mathbb{N}$. Let $s_i^m$ denote the $m$-th element of $S_i$ sorted in increasing order, and define $t_i^m$ similarly. If for all $i_1,i_2,...,i_p\in\mathbb{N}$, \[
		\sum_{k=1}^p s_k^{i_k}=\sum_{k=1}^p t_k^{i_k},
	\]
	then $X(\mathcal{S})$ is conjugate to $X(\mathcal{T})$.
\end{theorem*}

Both theorems are generalizations of existing results for $(S, S')$-gap shifts from \cite{SPrimegap}.  The former is a generalization of a theorem that states if $X(S,S')$ and $X(T,T')$ are conjugate $(S,S')$-gap shifts, then $\{\!\!\{S + S'\}\!\!\} = \{\!\!\{T + T'\}\!\!\}$.  The latter is a generalization of a theorem that assumes $\{\!\!\{S + S'\}\!\!\} = \{\!\!\{T + T'\}\!\!\}$ and gives specific cases for which $X(S,S')$ is conjugate to $X(T,T')$.

%At the end of this paper, we consider a more general class of $\mathcal{S}$-limited shifts, called generalized $\mathcal{S}$-limited shifts, in which we remove the restriction of the order in which the blocks may appear.  While we can extend a few of the results about mixing properties and cases in which a generalized $\mathcal{S}$-limited shift is an SFT or sofic, we are unable to reproduce our main results concerning entropy, intrinsic ergodicity and conjugacy due to the lack of a set which would replace the set $G_\mathcal{S}$ in the $\mathcal{S}$-limited shift setting.  

\section{Background and Definitions}
Let $\mathcal{A} = \{1, \ldots, p\}$ be an alphabet, and let $X$ be a shift space on $\mathcal{A}$.  Define $B_n(X) = \{ \omega_1 \ldots \omega_n : \omega_i \in \mathcal{A} \text{ for } 1 \leq i \leq n \}$ to be the collection of all words of length $n$ and let $\mathcal{L}(X) = \bigcup_{n \geq 1} B_n(X)$ denote the \textbf{language} of $X$.  If the shift space is understood in context, we will use $B_n$ and $\mathcal{L}$.  
% For $\omega\in B_n$, the \textbf{cylinder set} $\cylinderset{\omega}$ a subset of $\Omega$ such that \[\cylinderset{\omega}=\{\tau\in X : \tau_0\tau_1\cdots \tau_{n-1}=\omega_0\omega_1\cdots\omega_{n-1}\}.\]
A shift space $X$ is called \textbf{irreducible} if for all $\omega, \tau \in \mathcal{L}(X)$, there exists some $\xi  \in \mathcal{L}(X)$ such that $\omega \xi \tau \in \mathcal{L}(X)$.  A word $\xi \in \mathcal{L}(X)$ is called \textbf{synchronizing} if for any $\omega \xi, \xi \tau \in \mathcal{L}(X)$, the word $\omega \xi \tau \in \mathcal{L}(X)$.  An irreducible shift space $X$ with a synchronizing word is called a \textbf{synchronized system}.  A shift space $X$ is called \textbf{mixing} if for all $\omega, \tau \in \mathcal{L}(X)$, there exists some $N$ such that for each $n \geq N$, there is a $\xi \in B_n(X)$ such that $\omega \xi \tau \in \mathcal{L}(X)$.  
\medskip

Let $X$ denote the full shift on $\mathcal{A}$.  A \textbf{subshift of finite type (SFT)} is a subshift $X_F \subseteq X$ defined by a finite collection $F\subset \mathcal{L}(X)$ of forbidden words, meaning that $X_F=\{\omega \in X: \tau \text{ appears nowhere in } \omega \text{ for all } \tau \in F\}$. For any given SFT, the list of forbidden words define both an adjacency matrix and finite graph presentation that represent the shift space.  The finite graph presentation consists of finitely many vertices, and infinite walks correspond to infinite strings in the shift space.  Similarly, one can consider a broader class, called \textbf{sofic} subshifts, which includes all subshifts that have a finite graph presentation.  As an adjacency matrix can be defined from a finite graph, each sofic subshift has an associated adjacency matrix.  For more information on SFTs and sofic subshifts, see \cite{LM}.
\medskip

One of the simplest constructions of a subshift that is neither sofic nor an SFT on the alphabet $\mathcal{A} = \{0,1\}$ is the prime gap shift, which includes all strings in which any two adjacent 1s are separated by a prime number of 0s.  This shift space is an example of an $S$-gap shift.  
An \textbf{$\bm{S}$-gap shift} is a subshift of the form \[X_S=\overline{\{...10^{n_{-1}}10^{n_0}10^{n_1}1... : n_i\in S\}},\]
where $S \subseteq \mathbb{N}_0$.  While the prime gap shift provides an example of a non-sofic subshift, there exist $S$-gap shifts that are SFTs or sofic as well.  However, we cannot guarantee the existence of a finite graph presentation or adjacency matrix for a general $S$-gap shift, and hence new techniques must be used for exploring properties of these shift spaces.  A more general class of subshifts that includes all $S$-gap shifts is the $(S, S')$-gap shifts.  An $\bm{(S,S')}$\textbf{-gap shift} is a subshift on the alphabet $\mathcal{A} = \{0,1\}$ of the form: \[X(S,S')=\overline{\{...1^{m_{-2}}0^{n_{-1}}1^{m_{-1}}0^{n_0}1^{m_0}0^{n_1}1^{m_1}... : n_i\in S,m_i\in S'\text{ for all } i\in\mathbb{Z}\}},\]
where $S, S' \subseteq \mathbb{N}$.
Again, it is the case that $(S,S')$-gap shifts can be SFT, sofic, or neither.  To clarify some of these characterizations, we consider the following examples.

\begin{example}[Golden mean shift]
Let $\mathcal{A} = \{0,1\}$.  The golden mean shift is an SFT with forbidden word list $F = \{11\}$.  It has the following graph presentation and adjacency matrix:
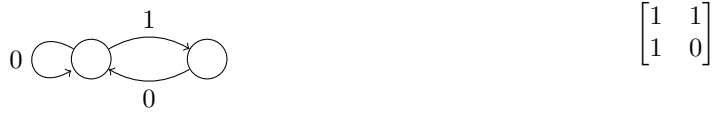
\begin{center}
\begin{multicols}{2}
\begin{tikzpicture}

\tikzset{vertex/.style = {shape=circle,draw,minimum size=1.5em}}
\tikzset{edge/.style = {->,> = latex'}}

\node[vertex] (A) {}
	edge[in=210,out=150,loop] node[auto,swap] {0} (A);
\node[vertex] (B) [right=of A] {}
	edge[<-, bend right] node[auto,swap] {1} (A)
    edge[->, bend left] node[auto] {0} (A);

\end{tikzpicture}

$\begin{bmatrix}
1 & 1 \\
1 & 0
\end{bmatrix}$
\end{multicols}
\end{center}

The golden mean shift is also an $S$-gap shift with $S = \mathbb{N}$.

\end{example}

\begin{example}[Even Shift]
Let $\mathcal{A} = \{0,1\}$.  The even shift is a sofic subshift that contains all infinite strings in which any two adjacent 1s are separated by an even number of 0s.  It has the following graph presentation and adjacency matrix:

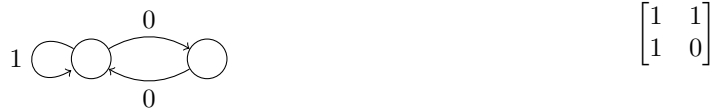
\begin{center}
\begin{multicols}{2}
\begin{tikzpicture}

\tikzset{vertex/.style = {shape=circle,draw,minimum size=1.5em}}
\tikzset{edge/.style = {->,> = latex'}}

\node[vertex] (A) {}
	edge[in=210,out=150,loop] node[auto,swap] {1} (A);
\node[vertex] (B) [right=of A] {}
	edge[<-, bend right] node[auto,swap] {0} (A)
    edge[->, bend left] node[auto] {0} (A);

\end{tikzpicture}

$\begin{bmatrix} 1&1 \\ 1&0 \end{bmatrix}$

\end{multicols}
\end{center}
  It is an $S$-gap shift with $S = \{2n : n \in \mathbb{N}_0\}$.  
\end{example}

For the remainder of this paper, unless otherwise specified, we fix our alphabet $\mathcal{A} = \{1,2, \ldots, p\}$.  
Building upon the ideas of an $S$-gap and $(S,S')$-gap shift, we arrive at out definition of an $\mathcal{S}$-limited shift.  As a preliminary step, we will introduce our collection of limiting sets $\mathcal{S} = \{S_1, \ldots, S_p\}$ with $S_i \subseteq \mathbb{N}$ for $1 \leq i \leq p$ and a collection of finite blocks 
\[G_\mathcal{S}=\{1^{m_1}2^{m_2}...p^{m_p} : m_i\in S_i\text{ for }1\leq i\leq p\}\]
called the \textbf{core set}.  An \textbf{$\bm{\mathcal{S}}$-limited shift} is a subshift on $\mathcal{A}$ defined by 
\[X(\mathcal{S})=\overline{\{\cdots x_{-1}x_0x_1\cdots : x_i\in G_\mathcal{S}\text{ for all } i \in \mathbb{Z}\}}.\]
In the case that any of the $S_i$ are infinite, we note that taking the closure is a non-trivial step in defining the subshift, as we now consider bi-infinite strings ending or beginning with an infinite string of a single letter.
\medskip

While most of our results concern $\mathcal{S}$-limited shifts, one can extend a few results to a more general setting.  So far, we have been restricting the order in which the blocks of letters can appear in $\mathcal{S}$-limited shifts.  Let $\mathcal{A}$ and $\mathcal{S}$ be as above.  A \textbf{generalized $\mathcal{S}$-limited shift}, denoted $X_\mathcal{S}$, is the closure of the set
\[ \{\ldots \omega_{-1}^{\alpha_{-1}}\omega_0^{\alpha_0}\omega_1^{\alpha_1} \ldots \mid \omega_i \in \mathcal{A}, \omega_i \neq \omega_{i+1}, \text{ and } \alpha_i \in S_{\omega_i} \text{ for all } i\}. \]
The first difference between $\mathcal{S}$-limited shifts and generalized $\mathcal{S}$-limited shifts is that the full shift on $p$ letters is a generalized $\mathcal{S}$-limited shift with $S_i = \mathbb{N}$ for $1 \leq i \leq p$.  However, the full shift is not an $\mathcal{S}$-limited shift.  The most significant difference between these two classes of shift spaces is the absence of an analogous core set $G_\mathcal{S}$ in the case of a generalized $\mathcal{S}$-limited shift.  
\medskip

As mentioned in the introduction, we will discuss the entropy of $\mathcal{S}$-limited shifts.  Although the definition of entropy is more complex in a general setting (see \cite{Walters} for details), the definition reduces nicely in the symbolic dynamics setting.  The \textbf{(topological) entropy} of a subshift $X$, denoted by $h(X)$,  is given by
\[h(X) = \lim_{n \to \infty} \frac{\log(B_n(X))}{n}.\]
Entropy can be used to define another class of subshifts.  A subshift $X$ is \textbf{almost sofic} if given $\varepsilon >0$, there exists an SFT $Y\subset X$ such that $h(Y) > h(X) - \varepsilon$.  We are using the definition from \cite{Peterson}, but analogous definitions are used in other sources, such as \cite{LM}.
\medskip
  
Entropy is widely used as a conjugacy invariant in the study of dynamical systems.  We say that two subshifts $X$ and $Y$ are conjugate if there exists a homeomorphism $\varphi:X \to Y$ such that $\varphi \circ \sigma = \sigma \circ \varphi$, where $\sigma$ denotes the shift map.  In searching for conjugacy between two shift spaces, we often look for an invertible sliding block code.  A \textbf{sliding block code} with memory $m$ and anticipation $n$ is a map $\varphi:X\to\mathcal{A}^{\mathbb{Z}}$ defined by $y=\varphi(x)$ with $y_i=\Phi(x_{[i-m,i+n]})$, where $\Phi:\mathcal{B}_{m+n+1}(X)\to\mathcal{A}$ and $x_{[i-m,i+n]} = x_{i-m}x_{i-m+1} \cdots  x_{i+n} \in \mathcal{B}_{m+n+1}(X)$.

%%Include example of sliding block code?

\section{Dynamical Properties}
For the remainder of this paper, we will fix our alphabet $\mathcal{A} = \{1, \ldots, p\}$, and let $\mathcal{S} = \{S_1, \ldots S_p\}$ with $S_i \subseteq \mathbb{N}$.  
We will use $X(\mathcal{S})$ to denote the $\mathcal{S}$-limited shift and $X_\mathcal{S}$ to denote a generalized $\mathcal{S}$-limited shift.  As we first turn our attention to some of the properties of $\mathcal{S}$-limited shifts, we first note some unique features of these subshifts.  
 In the case that $S_i = \mathbb{N}$ for each $1 \leq i \leq p$, the corresponding $\mathcal{S}$-limited shift is an SFT with forbidden word list $F = \{ nm :m\neq n, m \neq n+1, 1\leq n < p\} \cup \{pi : 1<i < p\}$.  Given that a maximal $\mathcal{S}$-limited shift is an SFT, we consider the following proposition, which gives necessary and sufficient conditions for which an $\mathcal{S}$-limited shift is an SFT.

\begin{proposition} \label{SFT limited}
$X(\mathcal{S})$ is an SFT if and only if $S_i$ is finite or cofinite for every $S_i\in\mathcal{S}$.
\end{proposition}

\begin{proof}  First, recall that we are restricting the order in which blocks may appear.  Hence, 
\[ F_0 = \{nm : m \neq n, m \neq n+1, 1 \leq n < p\} \cup \{ pi : 1<i < p\} \]
is forbidden for any $X( \mathcal{S})$.  Next, assume $S_i$ is finite.  Then, if $S_i$ is finite,
\[F_i=\{ai^nb:a,b\in\mathcal{A}\setminus\{i\},n\in\{1,2,...,\max S_i\}\setminus S_i\}\cup\{i^{1+\max S_i}\}\]
is a finite list of forbidden words associated with the set $S_i$. If, on the other hand, $S_i$ is cofinite, then
\[F_i=\{ai^nb:a,b\in\mathcal{A}\setminus\{i\},n\in\mathbb{N}\setminus S_i\}\]
is a finite list of forbidden words associated with the set $S_i$.  Hence, $F = \bigcup_{i=0}^n F_i$ is a finite list of forbidden words, and so $X(\mathcal{S})$ is an SFT. \\
Now, assume that $X(\mathcal{S})$ is an SFT.  Notice that if there exists an $S_i$ that is neither finite nor cofinite, then the forbidden word list $F_i$ associated with that set must be infinite.  Hence, each $S_i \in \mathcal{S}$ must be finite or cofinite.  
\end{proof}

\begin{proposition}  A generalized $\mathcal{S}$-limited shift $X_\mathcal{S}$ is an SFT if and only if every $S_i$ is finite or cofinite for every $S_i \in \mathcal{S}$.  
\end{proposition}
The proof follows exactly as in Proposition \ref{SFT limited}, except we can discard the forbidden word list $F_0$ since we are no longer restricting the order in which the blocks appear.  
\medskip

Next, we aim to find conditions for which an $\mathcal{S}$-limited shift will be sofic.  As a preliminary step, we introduce some terminology and an alternative definition for sofic subshifts.  Let $X$ be a subshift and $\omega \in \mathcal{L}(X)$.  The \textbf{follower set} of $\omega$, denoted $\mathcal{F}_X(\omega)$ is the set of the form $\mathcal{F}_X(\omega)=\{\tau\in \mathcal{L}(X) : \omega\tau\in \mathcal{L}(X)\}$.  A subshift $X$ is sofic if and only if it has a finite number of distinct follower sets.  For more information on sofic subshifts, see \cite{LM}.

We will also use the difference sets associated with each $S_i$.  Let 
\[\Delta(S_i)=\{s_0,d_1,d_2, d_3... : d_i=s_{i}-s_{i-1} \text{ for } i \geq 1\}\]
be the sequence of differences between consecutive entries of $S_i$ for $1 \leq i \leq p$.

\begin{proposition} \label{sofic limited}
$X(\mathcal{S})$ is sofic if and only if for every $S_i\in\mathcal{S}$, $\Delta(S_i)$ is eventually periodic.
\end{proposition}

\begin{proof}
First, notice that a follower set is determined by the last letter of the last block of letters in the finite word.  More specifically, for $\omega \in \mathcal{L}(X(\mathcal{S}))$, $i \in \mathcal{A}$, and any $n \in \mathbb{N}$, we can see that $F(\omega a i^n) = F(ai^n)$, where $a = i-1$ if $2 \leq i \leq p$ and $a = p$ if $i=1$.  Hence, we will only consider follower sets of words of the form $a i^n$ or $i$.\\
Now, assume that $\Delta(S_i)$ is eventually periodic.  If $S_i = \{s_0, s_1, \ldots\}$, then \newline $\Delta(S_i) = \{d_0, d_1, \ldots, d_{k-1}, \overline { m_1, \ldots, m_l}\}$.  Then, assuming $i \in \mathcal{A}$ and $a = i-1$ when $2 \leq i \leq p$ and $a=p$ when $i=1$, the follower sets can take three forms: 
\begin{align*}
&F(i), &1 \leq i \leq p \\
&F(ai^n), & 1\leq n \leq s_{k-1}\\
&F(ai^{s_{k+i-2}+j_i}), &0 \leq j_i \leq m_i-1 \text{ and } 1 \leq i \leq l
\end{align*}
Therefore, given our initial statement regarding the form of follower sets, there can only be at most finitely many follower sets for $X(\mathcal{S})$.  Thus, $X(\mathcal{S})$ is sofic.\\
Next, assume that $X(\mathcal{S})$ is sofic and so $X(\mathcal{S})$ has only finitely many follower sets.  Then, there must exist some integers $m$ and $q$ with $m < q$ such that $F(ai^m) = F(ai^q)$.  Choose the smallest such $q$ for which there exists such an $m$.  Now, let $S_i \cap (m,q] = \{n_1, \ldots, n_t\}$ and $s_{max} = \max \{ s \in S_i : s \leq m\}$.  Then,
\[\Delta(S_i) = \{s_0, s_1-s_0, \ldots, n_1-s_{max}, \overline{n_2-n_1, \ldots, n_t-n_{t-1}, q+n_1-n_t}\}, \]
and so $\Delta(S_i)$ is eventually periodic.  Since this argument is not dependent on the choice of $i$, we conclude that each $\Delta(S_i)$ must be eventually periodic.

\end{proof}

\begin{proposition} A generalized $\mathcal{S}$-limited shift $X_\mathcal{S}$ is sofic if and only if $\Delta(S_i)$ is eventually periodic for every $S_i \in \mathcal{S}$.
\end{proposition}
The proof follows from the proof of Proposition \ref{sofic limited} with one small modification.  When considering follower sets of the form $F(ai^n)$, we allow $a\in\mathcal{A} \setminus \{i\}$.  
\medskip

Next, we move on to properties involving the language of an $\mathcal{S}$-limited shift.  First, notice that all $\mathcal{S}$-limited shifts are irreducible and synchronized with synchronizing words of the form $p1$ or $a(a+1)$ where $1 \leq a < p$.  The following example highlights some of the challenges we face when determining which $\mathcal{S}$-limited shifts are mixing. 
\begin{example}
Let $\mathcal{A} = \{1,2\}$ and let $S_1 = S_2 = \{ 2n+1 : n \geq 0\}$.  Notice that $21, 12 \in \mathcal{L}(X(\mathcal{S}))$.  
If $\omega \in \mathcal{L}(X(\mathcal{S}))$ is a word such that $21\omega12 \in \mathcal{L}(X(\mathcal{S}))$, then $\ell(\omega) = 2l+1$ for some $l \geq 0$.  
Hence, for any $N \geq 1$ and $2n \geq N$, there exists no word $\omega \in B_{2n}(X(\mathcal{S}))$ such that $21 \omega 12 \in \mathcal{L}(X(\mathcal{S}))$.
\end{example}

\begin{proposition} \label{mixing}
An $\mathcal{S}$-limited shift $X(\mathcal{S})$ is mixing if and only if $\gcd\{s_1 + \cdots + s_p : s_i \in S_i\} = 1$.
\end{proposition}
\begin{proof}
First, we will assume that $X(\mathcal{S})$ is mixing.  Since $p1 \in \mathcal{L}(X(\mathcal{S}))$, then there exists some $N$ such that for all $n \geq N$, there is a word $\omega \in B_n(X(\mathcal{S}))$ such that $p1\omega p1 \in \mathcal{L}(X(\mathcal{S}))$.  
Since $\omega$ must be of the form $ 1^{s_1 - 1} 2^{s_2} \ldots p^{s_p} \tau_1 \tau_2 \ldots \tau_n 1^{s_1 '}2^{s_2'}\ldots (p-1)^{s_{p-1}'} p^{s_p'-1}$, where $\tau_k \in G_\mathcal{S}$ for $1 \leq k \leq n$ and $s_i, s_i' \in S_i$ for $1 \leq i \leq p$.  
Hence, there exist words of this form, say $\omega_1$ and $ \omega_2$, of lengths $n$ and $n+1$, respectively.  
Thus, there exist words $1\omega_1p$ and $1\omega_2p$ of lengths $n+2$ and $n+3$ (respectively), where both $1\omega_1p$ and $1\omega_2p$ consist of concatenated blocks from $G_\mathcal{S}$.  Hence, $\gcd\{ s_1 + \cdots + s_p : s_i \in S_i\} = 1$.  \\
Next, we assume $\gcd \{s_1+\cdots + s_p : s_i \in S_i\} =1$.  
Then, there exists some sufficiently large $N$ such that for all $n \geq N$, there exists a word of length $n$ of the form $\tau_1 \cdots \tau_{m}$, where $\tau_i \in G_\mathcal{S}$ for all $1 \leq i \leq m$.  Hence, $p\tau_1 \cdots \tau_m 1 \in B_{n+2}(X(\mathcal{S}))$.  Since $p1$ is synchronizing and $X(\mathcal{S})$ is irreducible, then $X(\mathcal{S})$ must be mixing.
\end{proof}

A generalized $\mathcal{S}$-limited shift does not have a core set analogous to the set $G_\mathcal{S}$ in the $\mathcal{S}$-limited shift setting.  To parallel our use of $G_\mathcal{S}$ in the previous proof, we can think of words from $X_\mathcal{S}$ in terms of non-repeating blocks, that is, we consider finite words of the form $a_1^{n_1} \tau_1 \tau_2 \ldots \tau_m a_2^{n_2}$, where $a_1, a_2 \in \mathcal{A}$, $n_1, n_2 \in \mathbb{N}$, and $\tau_i = a_{i1}^{s_{i1}}a_{i2}^{s_{i2}}\ldots a_{ik}^{s_{ik}}$, with $a_{ij} \neq a_{il}$ for $j \neq l$ and $s_{ij} \in S_{a_{ij}}$.  

\begin{proposition} 
A generalized $\mathcal{S}$-limited shift $X_\mathcal{S}$ is mixing if and only if \newline $\gcd\left \{\sum_{i=1}^k s_{a_i} : s_{a_i} \in S_{a_i}, a_i \neq a_j (\text{for } i\neq j), 2\leq k \leq p \right \} =1$.  
\end{proposition}

\begin{proof}
First, assume $X_\mathcal{S}$ is mixing.  Let $a_1$, $a_2 \in \mathcal{A}$ with $a_1 \neq a_2$.  Then, $a_1a_2 \in \mathcal{L}(X_\mathcal{S})$ and hence there exists $N$ such that for all $n \geq N$, there is a word $\omega \in B_n(X_\mathcal{S})$ such that $a_1 a_2 \omega a_1 a_2 \in \mathcal{L}(X_\mathcal{S})$. 
 Notice that $\omega$ must be of the form $\omega = a_2^{s_{a_1}-1} a_1^{s_{a_2}-1}$ or $\omega = a_2^{s_{a_1}-1} \tau_1 \tau_2 \ldots \tau_m a_1^{s_{a_2}-1}$ where $\tau_i = a_i^{s_i}$ with $s_i \in S_{a_i}$ for $1 \leq i \leq m$.  Hence, there exist words $\omega_1 \in B_n(X_\mathcal{S})$ and $\omega_2 \in B_{n+1}(X_\mathcal{S})$ of this form.  
Therefore, there exist words $a_2 \omega_1 a_1 \in B_{n+2}(X_\mathcal{S})$ and $a_2 \omega_2 a_1 \in B_{n+3}(X_\mathcal{S})$, and each word consists of concatenated blocks of $i$'s of length $s_i \in S_i$ where $i \in \mathcal{A}$.  Since $a_1$ and $a_2$ were chosen arbitrarily, then \newline $\gcd\left \{\sum_{i=1}^k s_{a_i} : s_{a_i} \in S_{a_i}, a_i \neq a_j (\text{for } i\neq j), 2\leq k \leq p \right \} =1$.
\medskip

Next, assume $\gcd\left \{\sum_{i=1}^k s_{a_i} : s_{a_i} \in S_{a_i}, a_i \neq a_j (\text{for } i\neq j), 2\leq k \leq p \right \} =1$.  Then, there exists some sufficiently large $N$ such that for all $n \geq N$, there exists a word $a_1^{s_{a_1}}a_2^{s_{a_2}} \ldots a_m^{s_{a_m}} \in B_n(X_\mathcal{S})$, where $a_i \in \mathcal{A}$, $a_i \neq a_{i+1}$ for $1 \leq i < m$ and $s_{a_1} \in S_{a_1}$.  Hence, $a_0  a_1^{s_{a_1}}a_2^{s_{a_2}} \ldots a_m^{s_{a_m}} a_{m+1} \in B_{n+2}(X_\mathcal{S})$, where $a_0 \neq a_1$ and $a_m \neq a_{m+1}$.  Notice that $a_i a_j$ is synchronizing for all $a_i \neq a_j$ with $a_i, a_j \in \mathcal{A}$.  Since $X_\mathcal{S}$ is irreducible, then $X_\mathcal{S}$ must be mixing.

\end{proof}

\section{Entropy}

In the remaining sections, we will only work with $\mathcal{S}$-limited shifts.  Our current methods would require significant modifications to be extended to the generalized $\mathcal{S}$-limited shift setting due to the fact that they rely on the existence of a good core set, $G_\mathcal{S}$. 
\medskip

Entropy is a conjugacy invariant that is often sought in symbolic dynamics.  While entropy of SFTs and sofic subshifts is well-understood (see \cite{LM}) due to the existence of adjacency matrices in these settings, entropy calculations can vary in more general shift space settings.  Entropy calculations exist for both $S$-gap shifts and $(S, S')$-gap shifts (see \cite{LM} and \cite{SPrimegap}).

\begin{theorem} \label{entropy}
Let $\mathcal{S}=\{S_1,S_2,...,S_p\}$ with $S_i\subseteq\mathbb{N}$ for $1\leq i \leq p$. Then, the entropy of $X(\mathcal{S})$ is $\log \frac{1}{\lambda}$, where $\lambda$ is the unique positive solution to $\displaystyle\sum_{\omega \in G_\mathcal{S}}x^{|\omega|}=1$.
\end{theorem}

\begin{proof}
Consider the generating function $H(z)=\displaystyle\sum_{m=1}^\infty (\#B_m)z^m$, where $\#B_m$ denotes the cardinality of  $B_m(X(\mathcal{S}))$. First, we claim that the radius of convergence of $H(z)$ is $e^{-h(X(\mathcal{S}))}$.

To see this, we will  show that $\displaystyle\lim_{m\to\infty}\sqrt[m]{|\#B_mz^m|}=1$ when $z=e^{-h(X(\mathcal{S}))}$. First, notice that
 \[
e^{-h(X(\mathcal{S}))} = e^{-\lim_{m\to\infty}\frac{\log\#B_m}{m}}     =\lim_{m\to\infty}\left(\frac{1}{\#B_m}\right)^{1/m} .     
\]

By letting $z = e^{-h(X(\mathcal{S}))}$, we obtain
\[\lim_{m\to\infty}\left(\#B_mz^m\right)^{1/m}=\lim_{m\to\infty} (\#B_m)^{1/m} \left(\frac{1}{\#B_m}\right)^{1/m}=1.\]
Hence, $e^{-h(X(\mathcal{S}))}$ must be the radius of convergence of $H(z)$.
Next, we define the following values which depend on $m$ and $k$: \[
A_m^k=\#\{\omega\in B_m(X(\mathcal{S})):\omega=\tau_1\tau_2\cdots\tau_k\text{, where }\tau_i\in G_\mathcal{S} \text{ for }1\leq i\leq k\}.
\]
Using this, we can define the collection of functions \[
F_k(z)=\sum_{m=1}^\infty A_m^kz^m.
\]
Notice that $F_1(z)=\displaystyle\sum_{\omega \in G_\mathcal{S}}z^{|\omega|}=1$.  It remains to show that the unique positive solution to $F_1(z)=1$ is equal to the radius of convergence of $H(z)$, so that $e^{-h(X(\mathcal{S}))} = \lambda$ and hence $\log\left ( \frac{1}{\lambda}\right) = h(X(\mathcal{S}))$.

First, we will show that $F_k(z)F_l(z)=F_{k+l}(z)$ for all $k,l\in\mathbb{N}$. Let $m\in\mathbb{N}$ and consider the coefficient on the $z^m$ term of $F_{k+l}(z)$, which corresponds to all words $\omega$ formed by concatenating $k+l$ words from $G_\mathcal{S}$ such that $\ell(\omega) = m$.  Set $\omega = xy$, where $x$ is composed of $k$ concatenated words from $G_\mathcal{S}$ and $y$ is composed of  $l$ concatenated words from $G_\mathcal{S}$. Then, we know that $\ell(y)=m-\ell(x)$. Therefore, the coefficient on $z^m$ in $F_{k+l}(z)$ is $\displaystyle\sum_{i=0}^nA_{i}^kA_{m-i}^l$, which is the coefficient on $z^m$ in $F_k(z)F_l(z)$. Therefore, $F_k(z)F_l(z)=F_{k+l}(z)$ as claimed. From this, we notice $\displaystyle\sum_{k=1}^\infty F_k(z) = \sum_{k=1}^\infty (F_1(z))^k$, which converges for all 
$z$ for which $F_1(z) < 1$, so its radius of convergence is equal to the positive value of $z$ for which $F_1(z)=1$.

Next, we look to establish a relation between $F_1(z)$ and $H(z)$. While it is clear that $\displaystyle\sum_{k\geq 1}A_m^k \leq\#B_m$, we find that words $\omega\in B_m(X(\mathcal{S}))$ can take three forms, where in all cases if some word $x$ is not specifically in $G_\mathcal{S}$, then it cannot have any subwords in $G_\mathcal{S}$:
\begin{enumerate}
\item $\omega = x$ where $|x|=n$
\item $\omega\in A_n^k$
\item $\omega = x\tau y$ where $|x|=i$, $|y|=j$, $|\tau|=n-i-j$ and $\tau$ is the concatenation of $k$ words in $G_\mathcal{S}$
\end{enumerate}

The first question that is raised is, how many allowable words of length $m$ can have no subwords in $G_\mathcal{S}$. The first thing to notice is that there must be fewer than $2p$ transitions from one letter to another of a different value, otherwise a substring of $G_\mathcal{S}$ must appear since our alphabet contains $p$ letters.
 Now, we can create an upper bound for the number of such words by summing over the number of transitions between different letters. There are at most $\displaystyle p\sum_{t=0}^{2p}\binom{m}{t}$ such words, where $p$ corresponds to the number of options for a first letter and $t$ corresponds to the positions for the transitions. 
Then, we can see that is bounded above by $2p^2 m^{2p}$. Now, we define the sequence $W_m$ to be the number of allowable words of length $m$ that have no subwords in $G_\mathcal{S}$, and based on this upper bound, we can use the root test to see that the radius of convergence of $\displaystyle\sum_{m=1}^\infty W_mz^m$ is at least 1.

Using this, we can construct the following inequality: \[
\sum_{k\geq 1}A_m^k\leq \#B_m\leq W_m+\sum_{k\geq 1}W_k\sum_{i\geq 0}W_i\sum_{j\geq 0}A_{m-i-j}^k.
\]

Using that chain of inequalities, we obtain the following inequalities: \[
\sum_{k\geq 1}F_k(z)\leq H(z)\leq \sum_{m\geq 1}W_mz^m+\sum_{k\geq 1}F_k(z)\left(\sum_{i\geq 0}W_iz^i\right)\left(\sum_{j\geq 0}W_jz^j\right).
\]

Here, we can see that $H(z)$ will converge if and only if $\displaystyle\sum_{k\geq 1}F_k(z)$ converges, since above we showed that happens only if $F_1(z) < 1$ and all of the other sums on the rightmost side of the inequality have a radius of convergence of 1 so we do not need to worry about their effect on convergence. Therefore, the radius of convergence of $H(z)$ is equal to the unique positive solution to $F_1(z)=1$, which we have already shown proves our original claim.
\end{proof}

Now that we have established entropy calculations for $\mathcal{S}$-limited shifts, we can classify all shifts of this type.  
\begin{proposition}
All $\mathcal{S}$-limited shifts are almost sofic.
\end{proposition}

\begin{proof}
Let $X(\mathcal{S})$ be an $\mathcal{S}$-limited shift with $\mathcal{S} =\{S_1, \ldots, S_p\}$.  
Consider the $\mathcal{S}$-limited shift $X(\mathcal{S}\vert_n)$ defined by $\mathcal{S}\vert_n = \{S_1\vert_n, S_2\vert_n, \ldots, S_p\vert_n\}$, where $S_i\vert_n$ consists of the first $n$ entries of $S_i$ for $1\leq i \leq p$. By Proposition \ref{SFT limited}, $X(\mathcal{S}\vert_n)$ is an SFT.   By Theorem \ref{entropy}, entropy of an $\mathcal{S}$-limited shift is given by $\displaystyle\sum_{\omega \in G_\mathcal{S}}x^{|\omega|}=1$, and so we can choose $n$ large enough such that $h(X(\mathcal{S}\vert_n))$ is arbitrarily close to $h(X(\mathcal{S}))$.

\end{proof}

As mentioned earlier, (topological) entropy has a simplified definition for symbolic dynamic spaces.  The measure-theoretic entropy of $\sigma$ is denoted by $h_\mu(\sigma)$ where $\mu \in \mathcal{M}_\sigma(X)$ and $\mathcal{M}_\sigma(X)$ denotes the space of shift invariant probability measures on $X$.  While we will not state the full definition of measure-theoretic entropy, we will establish the connection between $h(X)$ and $h_\mu(\sigma)$.  The variational principle gives the following relationship:
\[h(X) = \sup\{h_\mu(\sigma) : \mu \in \mathcal{M}_\sigma(X)\}.\]
A shift space $X$ is \textbf{intrinsically ergodic} if there exists a  unique measure $\mu \in \mathcal{M}_\sigma(X)$ that attains the supremum in the variational principle.  For more information on entropy and intrinsically ergodic systems, see \cite{Walters}.

\begin{theorem}
Every subshift factor of an $\mathcal{S}$-limited shift is intrinsically ergodic.
\end{theorem}
\begin{proof}
Consider the following decomposition of $\mathcal{L}(X(\mathcal{S}))$:
\begin{align*}
G_\mathcal{S}^* = &\{\tau_1\tau_2 \ldots \tau_n : \tau_i \in G_\mathcal{S} \text{ for } 1 \leq i \leq n\} \\
\mathcal{C}^P = &\{ l^{n} (l+1)^{s_{l+1}}\cdots p^{s_{p}} : 1< l \leq p, s_i \in S_i \text{ for } l < i \leq p, n \in \mathbb{N}\}\\
 \cup &\{ 1^{n} 2^{s_2}\cdots p^{s_{p}} : n \in \mathbb{N}\setminus S_1, s_i \in S_i \text{ for } 2 \leq i \leq p\} \\
\mathcal{C}^S = &\{ 1^{s_1} \cdots (k-1)^{s_{k-1}} k^n : 1 \leq k \leq p, s_i \in S_i \text{ for } 1\leq i <k, n \in \mathbb{N}\} \\
\cup &\{ 1^{s_1} 2^{s_2}\cdots (p-1)^{s_{p-1}}p^n : n \in \mathbb{N}\setminus S_p, s_i \in S_i \text{ for } 1 \leq i \leq p-1\}
\end{align*}
Notice that $\mathcal{L} = \mathcal{C}^P G_\mathcal{S}^* \mathcal{C}^S$, i.e. for every $\omega \in \mathcal{L}$, $\omega = \xi_P \xi_G \xi_S$, where $\xi_P \in \mathcal{C}^P$, $\xi_G \in G_\mathcal{S}^*$, and $\xi_S \in \mathcal{C}^S$.  In order to use the results from \cite{CT}, we show that this decomposition satisfies three properties:
\begin{enumerate}[(i)]

\item For any $\tau_1, \ldots, \tau_m \in G_\mathcal{S}^*$, $\tau_1\tau_2\ldots\tau_m \in G_\mathcal{S}^*$, so $G_\mathcal{S}^*$ has specification.  

\item We will show that $\tilde{h}(\mathcal{C}^P \cup \mathcal{C}^S) = 0$, where $\tilde{h}(\cdot)$ denotes the growth rate $\tilde{h}(X) = \frac{1}{n}\log(B_n(X))$. We claim that the cardinalities of both $\mathcal{C}^S$ and $\mathcal{C}^P$ are bounded above by a polynomial of degree $2^{p-2}$, where $p = \abs{\mathcal{A}}$.  We will calculate the upper bound by considering the case where $S_i = \mathbb{N}$ for all $1 \leq i \leq p$.
Notice that when $p=2$, there is exactly 1 block of $n$ 1's and $(n-1)$ blocks of the form $1^m 2^{n-m}$.  So, $\abs{B_n(\mathcal{C}^P)} = n$, and similarly $\abs{B_n(\mathcal{C}^S)} = n$. 

If we assume that $\abs{B_n(\mathcal{C}^P)} \approx n^{2^{p-2}}$ for $p>2$, notice that $\abs{B_n(\mathcal{C}^P)} \approx \left(n^{2^{p-2}}\right)^2 = n^{2^{p-1}}$ for $(p+1)$ symbols.  
Hence, our claim holds.  Therefore, we obtain 
\[ \tilde{h}(\mathcal{C}^P \cup \mathcal{C}^S) = \lim_{n \to \infty} \frac{1}{n}\log(n^{2^{p-2}}) = 0\].  

\item Given our decomposition, for $M\in\mathbb{N}$, define the collections of words $G_\mathcal{S}(M)$ to be \[G_\mathcal{S}(M) = \{ \xi_P \xi_G \xi_S : \xi_P \in \mathcal{C}^P, \xi_G \in G_\mathcal{S}^*, \xi_S \in \mathcal{C}^S, \abs{\xi_P} \leq M, \abs{\xi_S} \leq M\} .\]  Given some $\xi_P\xi_G\xi_S \in G_\mathcal{S}(M)$, notice that $\xi_P$ and $\xi_S$ determine the length of word $u$ and $v$ such that $u \xi_P\xi_G \xi_S v \in G_\mathcal{S}$.  Since there are only finitely many such $\xi_P$ and $\xi_S$, then there exists some $t$ such that there exists words $u$ and $v$ with $\abs{u}\leq t$ and $\abs{v} \leq t$ for which $u\xi_P \xi_G \xi_S v \in G_\mathcal{S}$.
\end{enumerate}
Since our decomposition satisfies these three properties, then by \cite{CT}, an $\mathcal{S}$-limited shift and all of its subfactors must be intrinsically ergodic.
\end{proof}

\section{Conjugacy}
Next, we address the question of conjugacy of $\mathcal{S}$-limited shifts through a generalization of the results in \cite{SPrimegap}. The benefit of restricting the order in which the blocks appear is that we can now split the elements of an $\mathcal{S}$-limited shift up into its building blocks from $G_\mathcal{S}$, since we know that every time we begin a run of $1$s a new element of $G_\mathcal{S}$ has begun, which is the only time that can happen. Therefore, we can consider elements of $\mathcal{S}$-limited shifts in terms of the elements of $G_\mathcal{S}$ they consist of and the order in which those elements appear.

In the following proof, we consider two $\mathcal{S}$-limited shifts, say $X(\mathcal{S})$ and $X(\mathcal{T})$.  We will denote the core sets associated with the shifts $X(\mathcal{S})$ and $X(\mathcal{T})$ by $G_\mathcal{S}$ and $G_\mathcal{T}$, respectively.

\begin{theorem}
Let $\mathcal{S}=\set{S_1,...,S_p}$ and $\mathcal{T}=\set{T_1,...,T_q}$ and suppose $X(\mathcal{S})$ is conjugate to $X(\mathcal{T})$. Then, for all $l\in\mathbb{N}$, $\abs{\set{x\in G_\mathcal{S} : |x|=l}}=\abs{\set{y\in G_\mathcal{T} : |y|=l}}$.
\end{theorem}

\begin{proof}
Let $X(\mathcal{S})$ be conjugate to $X(\mathcal{T})$. Suppose, to the contrary, that there exists some $l\in\mathbb{N}$ such that $\abs{\set{x\in G_\mathcal{S} : |x|=l}}\neq \abs{\set{y\in G_\mathcal{T} : |y|=l}}$. 
In particular, let $l$ be the smallest such length and suppose without loss of generality that $\abs{\set{x\in G_\mathcal{S} : |x|=l}} > \abs{\set{y\in G_\mathcal{T} : |y|=l}}$. 
Now, consider the number of points of period $l$ in $X(\mathcal{S})$ and $X(\mathcal{T})$. There must be the same number of periodic points that are built using elements of $G_\mathcal{S}$ and $G_\mathcal{T}$ of length less than $l$ because we can construct a length-preserving bijection between the elements of $G_\mathcal{S}$ and $G_\mathcal{T}$ that are shorter than $l$ letters. 
However, there will be $\abs{\set{x\in G_\mathcal{S} : |x|=l}}$ elements in $X(\mathcal{S})$ of the form $\bar{x}$ where $x\in G_\mathcal{S}$ and $|x|=l$ and $\abs{\set{y\in G_\mathcal{T} : |y|=l}}$ elements in $X(\mathcal{T})$ of the form $\bar{y}$ where $y\in G_\mathcal{T}$ and $|y|=l$, and those consist of the remaining words of period $l$ in the two subshifts respectively. 
Therefore, there are more words of period $l$ in $X(\mathcal{S})$ than there are in $X(\mathcal{T})$. However, conjugacy preserves the number of periodic points in a subshift, so we have a contradiction. Thus, for all $l\in\mathbb{N}$, we must have $\abs{\set{x\in G_\mathcal{S} : |x|=l}}=\abs{\set{y\in G_\mathcal{T} : |y|=l}}$.
\end{proof}

As shown in \cite{SPrimegap} for an $(S,S')$-gap shift (that is, the case where $p=q=2$), even though this is necessary, it is far from sufficient. Next, we can provide a sufficient case for sufficiency which is a generalization of their result for $(S,S')$-gap shifts.

\begin{theorem} \label{limited conjugacy}
	Let $\mathcal{S}=\set{S_1,...,S_p}$ and $\mathcal{T}=\set{T_1,...,T_p}$ where $S_i,T_i\subseteq\mathbb{N}$ for $1\leq i \leq p$. Let $s_i^m$ denote the $m$-th element of $S_i$ sorted in increasing order, and define $t_i^m$ similarly. If for all $i_1,i_2,...,i_p\in\mathbb{N}$, \[
		\sum_{k=1}^p s_k^{i_k}=\sum_{k=1}^p t_k^{i_k},
	\]
	then $X(\mathcal{S})$ is conjugate to $X(\mathcal{T})$.
\end{theorem}

In order to prove this theorem, however, we first need the following result to allow us to---as done in the case for sufficiency---think in terms of elements of $G_\mathcal{S}$.

\begin{definition}
Suppose $\mathcal{S}=\set{S_1,...,S_p}$ and $\mathcal{T}=\set{T_1,...,T_q}$. Then, we say the sliding block code $\varphi:X(\mathcal{S})\to X(\mathcal{T})$ induced by the block map $\Phi$ of memory $m$ and anticipation $n$ \textbf{induces} $\psi:G_\mathcal{S}\to G_\mathcal{T}$ if for all $x$ and $y$ of length $m$ and $n$ respectively and for all $z\in G_\mathcal{S}$, $\varphi(xzy)=\psi(z)$.
\end{definition}

\begin{example}[Induced Conjugacy]\label{induced conjugacy example}

Let $\mathcal{S}=\set{S_1,S_2,S_3}$ where $S_1=\mathbb{N}$, $S_2=\set{2n : n\in \mathbb{N}}$, and $S_3=\set{3,5}$. Also, let $\mathcal{T}=\set{T_1,T_2,T_3}$ where $T_1=\mathbb{N}$, $T_2=\set{2n+1 : n\in\mathbb{N}}$, and $T_3=\set{2,4}$. If we say $s_i^m$ and $t_i^m$ are the $m$th smallest element of $s_i$ and $t_i$ respectively, we notice that $s_1^m=t_1^m$, $s_2^m+1=t_2^m$, and $s_3^m-1=t_3^m$. Then, we can see that $\psi:G_\mathcal{S}\to G_\mathcal{T}$ where $\psi(1^n2^m3^p)=1^n2^{m+1}3^{p-1}$ is a conjugacy. If we take the block code $\Phi:\set{1,2,3}^2\to\set{1,2,3}$ with memory 1 and anticipation 0 where \[
\Phi(xy)=\begin{cases}
x \text{ if }xy=23 \\
y \text{ otherwise}
\end{cases}
\]
we can see that the sliding block code induced by $\Phi$ will induce $\psi$.
\end{example}

Not all elements of $X(\mathcal{S})$, however, will be the concatenation of elements of $G_\mathcal{S}$. Specifically, that will be the case when some element of $X(\mathcal{S})$ begins or ends with an infinite block of the same letter. Therefore, before we can speak about the action of sliding block codes that induce bijections from some $G_\mathcal{S}$ to some other $G_\mathcal{T}$, we must consider what happens when an element it would act on begins or ends with infinitely many of the same letter.

\begin{lemma} \label{induced infinite block}
Let $\mathcal{S}=\set{S_1,...,S_p}$ and $\mathcal{T}=\set{T_1,...,T_q}$ and $\varphi:X(\mathcal{S})\to \set{1,2,...,q}^\mathbb{Z}$ be a sliding block code induced by the block map $\Phi$ of memory $m$ and anticipation $n$ that induces a bijection $\psi:G_\mathcal{S}\to G_\mathcal{T}$. Also, let $I=\set{i\in\mathbb{N} : \abs{S_i}=\abs{\mathbb{N}}}$ and $J=\set{j\in\mathbb{N} : \abs{T_j}=\abs{\mathbb{N}}}$. Then, $\pi:I\to \set{1,2,...,q}$ where for all $i\in I$, $\Phi(i^r)=\pi(i)$ for $r=m+n+1$ must have image $J$.
\end{lemma}

\begin{proof}
Let $\mathcal{S}=\set{S_1,...,S_p}$ and $\mathcal{T}=\set{T_1,...,T_q}$ and $\varphi:X(\mathcal{S})\to \set{1,2,...,q}^\mathbb{Z}$ be a sliding block code induced by the block map $\Phi$ of memory $m$ and anticipation $n$ that induces a bijection $\psi:G_\mathcal{S}\to G_\mathcal{T}$, and set $r=m+n+1$. Also, let $I=\set{i\in\mathbb{N} : \abs{S_i}=\abs{\mathbb{N}}}$ and $J=\set{j\in\mathbb{N} : \abs{T_j}=\abs{\mathbb{N}}}$. From this, we can define a function $\pi:\set{1,2,...,p}\to\set{1,2,...,q}$ by $\pi(i)=\Phi(i^r)$. First we wish to show for any $j\in J$ there must exist some $i\in I$ so that $\pi(i)=j$, after which we will show $\pi(I)\subseteq J$.

Assume, to the contrary, that there exists $j \in J$ such that there does not exist any $i\in I$ so that $\pi(i)=j$. Now, let $x$ and $y$ be strings over $\set{1,2,...,p}$ of length $m$ and $n$ respectively, and consider the longest block of $j$s that can occur in the image of $\varphi(xzy)$ where $z\in G_\mathcal{S}$ with any block code $\Phi_j$ with memory $m$ and anticipation $n$ where $\Phi_j(i^r)\neq j$ for all $i\in I$. This will occur when $\varphi_j(xzy)=j^{|z|}$, so let us consider the largest $z$ where this can be the case. For each letter in $\set{1,2,...,p}\setminus I$, there is a maximum number of consecutive instances of that letter that can occur in any $z\in G_\mathcal{S}$. Clearly, for the $z$ we are constructing, we wish to have as many of each of these letters as possible. Let us say that $s$ is the largest of all of those values, then the length of $z$ contributed by these letters is at most $p*s$. For the letters in $I$, of which there are at most $p$, we can have at most $r-1$ consecutive instances of that letter. Therefore, the contribution to the length of $z$ from letters in $I$ is less than $p*r$. Thus, we cannot have more than $p(r+s)$ consecutive $j$s in the image of any word in $G_\mathcal{S}$ under $\varphi_j$. Then, since $\Phi_j$ will output a $j$ whenever $\Phi$ does, there cannot be more than $p(r+s)$ consecutive $j$s in the image of any word in $G_\mathcal{S}$ under $\varphi$. However, $T_j$ is infinite and $p(r+s)$ is finite, so there exists a word in $G_\mathcal{T}$ with more than $p(r+s)$ consecutive $j$s. Thus, $\varphi$ cannot induce a bijection from $G_\mathcal{S}$ to $G_\mathcal{T}$, giving us a contradiction. Thus, $\pi(I)\supseteq J$.

Now, suppose that $i\in I$. We wish to show $\pi(i)\in J$. To see this, let $k$ be the maximum value in any set in $\mathcal{T}$ that is not infinite, and fix some $i\in I$. Because $i\in I$, we know that $S_i$ is infinite, so there must be some value $l\in S_i$ where $l > k + r$. Therefore, there must exist $w\in G_\mathcal{S}$ that has $l$ consecutive $i$s. Then, $\varphi(w)$ has a block of $l-r  > k$ consecutive $\pi(i)$s. Therefore, $\pi(i)\in J$ because otherwise $\varphi(w)$ could not be in $G_\mathcal{T}$, which would contradict the fact that $\varphi$ induces a bijection from $G_\mathcal{S}$ to $G_\mathcal{T}$. Thus $\pi(I)\subseteq J$. Combining that with the result above that $\pi(I)\supseteq J$, we have $\pi(I)=J$.

\end{proof}

Using Lemma \ref{induced infinite block}, we can establish a link between sliding block codes that induce bijections and conjugacies.

\begin{lemma} \label{induced bijection}
Let $\mathcal{S}=\set{S_1,...,S_p}$ and $\mathcal{T}=\set{T_1,...,T_q}$ and $\varphi:X(\mathcal{S})\to \set{1,2,...,q}^\mathbb{Z}$ be a sliding block code that induces a bijection $\psi:G_\mathcal{S}\to G_\mathcal{T}$. Then, $\varphi$ is a conjugacy from $X(\mathcal{S})$ to $X(\mathcal{T})$.
\end{lemma}

\begin{proof}
Let $\mathcal{S}=\set{S_1,...,S_p}$ and $\mathcal{T}=\set{T_1,...,T_q}$ and suppose $\varphi:X(\mathcal{S})\to \set{1,2,...,q}^\mathbb{Z}$ is a sliding block code that induces a bijection $\psi:G_\mathcal{S}\to G_\mathcal{T}$. 
 First, we will show that the image of $\varphi$ is $X(\mathcal{T})$. 
Let $x\in X(\mathcal{S})$, and suppose that it neither begins nor ends with an infinite block of the same letter. Then, using the definition of an $\mathcal{S}$-limited shift, we can rewrite $x=\cdots x_{-1}x_0x_1\cdots$ where each $x_i\in G_\mathcal{S}$. Because $\varphi$ induces $\psi$, we can see $\varphi(x)=\cdots\psi(x_{-1})\psi(x_0)\psi(x_1)\cdots$. Then, since each $\psi(x_i)\in G_{\mathcal{T}}$, we have $\varphi(x)\in X(\mathcal{T})$.

Now, we must handle the case where $x$ begins and/or ends with an infinite block of the same letter. Suppose $x\in X(\mathcal{S})$ ends with an infinite block of the letter $j$. Then, through possibly reindexing the $x_i$s, we can say $x=\cdots x_{-2}x_{-1}x_0 y j^\mathbb{N}$ where each $x_i\in G_\mathcal{S}$, $y=1^{n_1}2^{n_2}\cdots (j-1)^{n_{j-1}}$ where $n_i\in S_i$ for all $i\in\set{1,2,...,j-1}$. Then, because $x$ ends with infinitely many $j$s, $S_j$ must be infinite, so there must be some $z\in G_\mathcal{S}$ so that $z$ begins $yj^r$ where $r$ is 1 plus the sum of the memory and anticipation of $\Phi$. Then, by Lemma \ref{induced infinite block}, if we say $\Phi(j^r)=k$, we know that we can have a word in $X(\mathcal{T})$ end in a block of infinitely many $k$s. Therefore, we have that $\varphi(x)=\cdots\psi(x_{-2})\psi(x_{-1})\psi(x_0)wk^\mathbb{N}$ where $w$ is the first $|y|$ letters of $\psi(z)$, from which we can see $\varphi(x)\in X(\mathcal{T})$. All of the remaining cases can be shown using the same approach of breaking as much of $x$ as possible up into elements of $G_\mathcal{S}$, and then viewing the remaining part as a substring of an element of $G_\mathcal{S}$ that begins and/or ends with a sufficiently long string of the letter that begins and/or ends $x$.  Then, we can use the fact that our sliding block code induces a bijection to see what it will map that word to, and we find that it is in $X(\mathcal{T})$.
Therefore, $\varphi(x) \in X(\mathcal{T})$ for all $x \in X(\mathcal{S})$.

Now, we will show that $\varphi:X(\mathcal{S})\to X(\mathcal{T})$ is a bijection. For that, first we will show that $\varphi$ is onto. Suppose $y\in X(\mathcal{T})$. Again, we first consider the case where $y=\cdots y_{-1}y_0y_1\cdots$ where each $y_i\in G_{\mathcal{T}}$. Then, since $\psi$ is a bijection, we can construct $x\in X(\mathcal{S})$ where $x=\cdots x_{-1}x_0x_1\cdots$ where for every $i\in\mathbb{Z}$, $x_i=\psi^{-1}(y_i)$. Then, using the fact that $\varphi$ induces $\psi$, we can see $\varphi(x)=y$.

Suppose now that $y\in X(\mathcal{T})$ begins and/or ends with an infinite block of the same letter---and as above, without loss of generality, say $y=\cdots y_{-2}y_{-1}y_0zj^\infty$ where each $y_i\in G_\mathcal{T}$ and $z=1^{n_1}2^{n_2}\cdots (j-1)^{n_{j-1}}$ where $n_i\in T_i$ for all $i\in\set{1,2,...,j-1}$. Now, for each $l\in \mathbb{N}$, let $w_l\in G_\mathcal{T}$ so that $w_l$ begins $yj^{l+r}$---note that this exists for every $l$ because this can be followed by more $j$s and $T_j$ is infinite. We wish to find a sufficiently large $l$ so the first $|yj^l|$ letters of $\psi^{-1}(w_l)$ ends with $i^r$ where $\Phi(i^r)=j$ with $S_i$ infinite. If we let $l_0=q(r+m)$ where $m$ is the largest element of any finite set in $\mathcal{T}$, then by similar reasoning to that employed in the proof of Lemma \ref{induced infinite block}, this must be the case for $w_{l_0}$. Thus, if $x=\cdots x_{-2}x_{-1}x_0 w i^\infty$ where $x_i=\psi^{-1}(y_i)$ for all $i\leq 0$, and $w$ is the first $|yj^l|$ letters of $\psi^{-1}(w_{l_0})$, then $\varphi(x)=y$. Therefore, $\varphi$ is onto.

Next, suppose $x,y\in X(\mathcal{S})$ with $x\neq y$. There are three ways in which this could happen: either $x$ and $y$ consist of a different sequence of words from $G_\mathcal{S}$, $x=\sigma^n(y)$ for some $n\neq 0$, or $x$ and $y$ consist of the concatenation of the same complete elements of $G_\mathcal{S}$ but are not equal (and therefore must begin and/or end with infinite blocks of some letter). In the second case, $\varphi(x)=\varphi(\sigma^n(y))=\sigma^n(\varphi(y))$. Thus, for us to have $\varphi(x)=\varphi(y)$, then $\varphi(x)$ must have a period that divides $n$. However, this means that $x$ and $\varphi(x)$ would have different periods because $x\neq y$ so $x$ cannot have a period that divides $n$, which gives us a contradiction. Therefore, if $x=\sigma^n(y)$, then $\varphi(x)\neq \varphi(y)$. If $x$ and $y$ consist of a different sequence of words from $G_\mathcal{S}$, then, because $\varphi$ induces the bijection $\varphi$, we must have $\varphi(x)$ and $\varphi(y)$ consist of a different sequence of words from $G_{\mathcal{T}}$. Thus, we must have $\varphi(x)\neq \varphi(y)$. Finally, suppose that $x$ and $y$ consist of the concatenation in the same order of the same complete elements of $G_\mathcal{S}$ but still $x\neq y$. Suppose, without loss of generality, that their difference comes after the end of the rightmost complete element of $G_\mathcal{S}$, and say that the parts of the words after the last complete element of $G_\mathcal{S}$ until there is an infinite block of some letter are $x'$ and $y'$ respectively, and that the letters repeated infinitely many times are $i$ and $j$ respectively. Notice that we must have $x'\neq y'$. Therefore, there must exist words $x''$ and $y''$ in $G_\mathcal{S}$ so that $x''$ begins with $x'i^r$ and $y''$ begins with $y'j^r$, and clearly $x_1\neq y_1$. We can have it so that $x''$ and $y''$ have the same number of every letter following the smaller of $i$ and $j$. Then, since $\psi$ is induced by a sliding block code, we can see that $\psi(x'')$ and $\psi(y'')$ must disagree on the section that is the image of the part up to and including the block of the smaller of the $i$s and $j$s. We can also see that the section of $\psi(x'')$ that is the image of the part of $x''$ up to the block of $i$s and the section of $\psi(y'')$ that is the image of the part of $y''$ up to the block of $j$s must be in $\varphi(x)$ and $\varphi(y)$ respectively, so we must have $\varphi(x)\neq \varphi(y)$. Therefore, we have $\varphi(x)=\varphi(y)$ if and only if $x=y$, so $\varphi$ is one-to-one. Thus, $\varphi$ is an invertible factor map, so it is a conjugacy.
\end{proof}

Before proving Theorem \ref{limited conjugacy}, we will introduce one last definition to keep our notation as simple as possible.

\begin{definition}
Let $\mathcal{S}=\set{S_1,...,S_p}$ and $\mathcal{T}=\set{T_1,...,T_q}$ and suppose $\varphi:X(\mathcal{S})\to X(\mathcal{T})$ is a sliding block code. Then, we will say that a \textbf{transition point} is any index $i$ so that, if $\varphi(x)=y=\cdots y_{-1}y_0y_1\cdots$ (where here $y_j \in \mathcal{A}$ for all $j$), $y_i\neq y_{i+1}$. If $y_i+1=y_{i+1}$, we will call this an internal transition point, and if $y_{i+1}=1$, we will call this an external transition point.
\end{definition}

The motivation for the distinction between internal and external transition points comes from thinking about elements of $X(\mathcal{S})$ as the concatenation of words from $G_\mathcal{S}$, since the internal transition points will be those within a word in $G_\mathcal{S}$ and the external transition points will be the last index within one word in $G_\mathcal{S}$, and so will be the transition between that word and the next one. Using this, we can finally turn to our proof.

\begin{proof}[Proof of Theorem \ref{limited conjugacy}]
Let $\mathcal{S}=\set{S_1,...,S_p}$ and $\mathcal{T}=\set{T_1,...,T_p}$ with $S_i,T_i\subseteq\mathbb{N}$ for $1\leq i \leq p$. Let $s_i^m$ denote the $m$th element of $S_i$ sorted in increasing order, and define $t_i^m$ similarly. Suppose that for all $i_1,i_2,...,i_p\in\mathbb{N}$, \[
	\sum_{k=1}^p s_k^{i_k}=\sum_{k=1}^p t_k^{i_k}.
\]

\noindent As an outline of the rest of the proof, we will first construct a bijection $\psi:G_{\mathcal{S}}\to G_\mathcal{T}$. Then, we will construct a block map $\Phi$ that induces a sliding block code $\varphi$ so that $\varphi$ induces $\psi$. Finally, from there, we will use Lemma \ref{induced bijection} to show that $\varphi$ is a conjugacy.

Let $x\in G_\mathcal{S}$. Then, by the definition of $G_\mathcal{S}$, we can write $x=1^{n_1}2^{n_2}\cdots p^{n_p}$ where each $n_i\in S_i$. Using the indexing of each $S_i$, we can rewrite that as $x=1^{s_1^{i_1}}2^{s_2^{i_2}}\cdots p^{s_p^{i_p}}$. Using that parameterization, we can construct $\psi:G_{\mathcal{S}}\to G_\mathcal{T}$ by letting $\psi(1^{s_1^{i_1}}2^{s_2^{i_2}}\cdots p^{s_p^{i_p}})=1^{t_1^{i_1}}2^{t_2^{i_2}}\cdots p^{t_p^{i_p}}$, which will be our bijection. 
Notice that, by the assumption that the sum of elements from the $S_i$s and $T_i$s with the same indices will be equal in in the statement of the theorem, for all $x\in G_\mathcal{S}$, $\abs{x}=\abs{\psi(x)}$, which is one of the properties of $X(\mathcal{S})$ and $X(\mathcal{T})$ that will allow us to construct a sliding block code that induces $\psi$.

The other main property we will use also comes from that assumption, which is that for all $j$, there exists some $n_j\in\mathbb{Z}$ so that, for all $i$, $s_j^i+d_j=t_j^i$. Specifically, we can see that $d_j=t_j^1+s_j^1$ by subtracting the sum from the statement of the sum where all the indices are 1 from the sum where all the indices are 1 except for the index for the $j$th set, and that index is $i$. Using this, we can rewrite $\psi$ so that it is expressed solely in terms of the sets $S_j$ and the constants $d_j$. In order to do that, we notice that $1^{t_1^{i_1}}2^{t_2^{i_2}}\cdots p^{t_p^{i_p}}$ can be rewritten as $1^{s_1^{i_1}+d_1}2^{s_2^{i_2}+d_2}\cdots p^{s_p^{i_p}+d_p}$. Therefore, we can say $\psi(1^{s_1^{i_1}}2^{s_2^{i_2}}\cdots p^{s_p^{i_p}})=1^{s_1^{i_1}+d_1}2^{s_2^{i_2}+d_2}\cdots p^{s_p^{i_p}+d_p}$.

Now, we can use this to figure out where the transition points in $x\in X(\mathcal{S})$ must be in any sliding block code that induces $\psi$. Because every element of $G_\mathcal{S}$ is mapped to an element of $G_\mathcal{T}$ of the same length, the external transition points must be the indices $i$ such that $x_i=n$ and $x_{i+1}=1$. Turning to the internal transition points, we can assume that we are only looking at some word $y\in G_\mathcal{S}$ because internal transition points only occur completely within one of those words. For any such $y$, there will be $p-1$ internal transition points: the change from $k$s to $(k+1)$s for all $k\in\set{1,2,...,p-1}$. Then, from the parameterization of $\psi$ in terms of only the $S_j$s and the $d_j$s, we can see that the internal transition point from $k$ to $k+1$ will occur at the letter indexed $\displaystyle\sum_{j=1}^k s_j^{i_j} + \sum_{j=1}^k d_j$. In terms of creating a sliding block code, the value that is of more importance than that, however, is the distance between the change from $k$s to $(k+1)$s in a word in $G_\mathcal{S}$ and the transition point that switches from $k$ to $k+1$ in the image of that word under $\psi$. Using the previous indexing of the transition points, we can see that distance will be $\displaystyle\sum_{j=1}^k d_j$, which we will refer to as $r_k$. Using those, we will define the value $r=\displaystyle 1+\max_k r_k$, which will be both the memory and anticipation of the block map $\Phi$ we create.

To construct $\Phi:B_{2r+3}(X(\mathcal{S}))\to\set{1,2,...,p}$, suppose $x\in B_{2r+3}(X(\mathcal{S}))$ where \newline$x=x_{-r-1}x_{-r}\cdots x_0\cdots x_{r}x_{r+1}$. First, find the transition point as described in the previous paragraph that is the closest to $x_0$, breaking ties to favor transition points that have negative indices in $x$. In order to do so, using the index offsets described, there are three cases for this we must address:

\begin{enumerate}

\item If there does not exist a transition point, then we will say $\Phi(x)=x_0$, because we know all the transition points are far enough from $x_0$ so that in $\Psi$ (as well as any sliding block code that induces $\Psi$), we must have that the letters with that index are the same in both the domain and the image.

\item If the closest transition point comes before $x_0$, suppose it is a transition point from $k$ to $l$, where $l=k+1$ for any internal transition point and $l=1$ for any external transition point. Then, we know that in the image of any word in $X(\mathcal{S})$ of which $x$ is a substring, and in any sliding block code inducing $\Psi$, the letter in the image of that word with the same index as $x_0$ must be $k+1$, so $\Phi(x)=l$.

\item If the closest transition point comes on or after $x_0$, suppose it is a transition point from $k$ to $l$, where $l=k+1$ for any internal transition point and $l=1$ for any external transition point. Then, for similar reasoning as in the previous case, $\Phi(x)=k$.

\end{enumerate}

Now, let $\varphi$ be the sliding block code induced by $\Phi$. By the construction of $\Phi$, we can see $\varphi$ will induce $\psi$. Therefore, by Lemma \ref{induced bijection}, $\varphi$ is a conjugacy from $X(\mathcal{S})$ to $X(\mathcal{T})$.

\end{proof}

Having established this proof, we can use that same technique to create a sliding block code that induces a bijection between the core sets of two specific $\mathcal{S}$-limited shifts.

\begin{example}[$\mathcal{S}$-limited Shift Conjugacy]

As in Example \ref{induced conjugacy example}, let $\mathcal{S}=\set{S_1,S_2,S_3}$ where $S_1=\mathbb{N}$, $S_2=\set{2n : n\in \mathbb{N}}$, and $S_3=\set{3,5}$, and also $\mathcal{T}=\set{T_1,T_2,T_3}$ where $T_1=\mathbb{N}$, $T_2=\set{2n+1 : n\in\mathbb{N}}$, and $T_3=\set{2,4}$. We can see that that $X(\mathcal{S})$ and $X(\mathcal{T})$ satisfy the conditions for Theorem \ref{limited conjugacy}, so $X(\mathcal{S})$ is conjugate to $X(\mathcal{T})$. In the previous example, we constructed a conjugacy $\psi:G_\mathcal{S}\to G_\mathcal{T}$, and then constructed a sliding block code that induced that conjugacy. While we could use that, along with Lemma \ref{induced bijection}, to show $X(\mathcal{S})$ is conjugate to $X(\mathcal{T})$, with Theorem \ref{limited conjugacy}, we can show the conjugacy by just noticing $s_1^m=t_1^m$, $s_2^m+1=t_2^m$, and $s_3^m-1=t_3^m$.

\end{example}

\section{Acknowledgements}
We are grateful to Do\u{g}an \c{C}\"{o}mez for his helpful comments on this work.

\bibliographystyle{plain}
\bibliography{s-lim-arxiv}

\end{document}